\documentclass[a4paper,10pt]{article}
\usepackage[centertags]{amsmath}
\usepackage{amsfonts}
\usepackage{amssymb}
\usepackage{amsthm}
\usepackage{epsfig}
\usepackage{setspace}
\usepackage{ae}
\usepackage{eucal}
\usepackage[usenames]{color}%
\usepackage{graphicx}

\theoremstyle{plain}
\newtheorem{thm}{Theorem}[section]
\newtheorem{lem}[thm]{Lemma}

\theoremstyle{definition}

\theoremstyle{remark}

\newtheorem{rem}[thm]{Remark}

\renewcommand{\div}{\operatorname{div}}

\title{The dilemma of turbulence modelling}
\author{J\"org Kampen }

\begin{document}

\maketitle

\begin{abstract}
A new construction technique of multiple solutions  of the Euler equation in strong spaces is introduced which reveals the relationship to multiple Navier Stokes equation solutions with special force terms while avoiding viscosity limit constructions. 
This shows that severe restrictions have to be imposed on time dependent external force terms in order to obtain uniqueness for the Navier Stokes equation Cauchy problem. Such restrictions are imposed in the statement of the so-called millenium problem. Minimal turbulence models should arguably incorporate weaker force terms in order to account for boundary conditions and forces. However, we show that  models of this type which have been proposed recently, do not have a unique solution. This lack of determinism of mimimal turbulence models indicates a dilemma: either models are too simple to capture turbulence but may have unique smooth solutions, or there is a modelling gap as the model does not determine a unique solution. 
\end{abstract}

\section{Introduction}
There have been many attempts and contibutions recently in order to make progress concerning the Cauchy problem 
\begin{equation}\label{Navleray}
\left\lbrace \begin{array}{ll}
\frac{\partial v_i}{\partial t}-\nu\sum_{j=1}^n \frac{\partial^2 v_i}{\partial x_j^2} 
+\sum_{j=1}^n v_j\frac{\partial v_i}{\partial x_j}=f_i\\
\\ \hspace{1cm}+\int_{{\mathbb R}^n}\left( \frac{\partial}{\partial x_i}K_n(x-y)\right) \sum_{j,m=1}^n\left( \frac{\partial v_m}{\partial x_j}\frac{\partial v_j}{\partial x_m}\right) (t,y)dy,\\
\\
\mathbf{v}(0,.)=\mathbf{h},
\end{array}\right.
\end{equation}
to be solved for the velocity $\mathbf{v}=\left(v_1,\cdots ,v_n \right)^T$ on the domain $\left[0,\infty \right)\times {\mathbb R}^n$, where the symbol ${\mathbb R}$ denotes the field of real numbers, $\nu >0$ is a viscosity constant, $K_n$ 
is the Laplacian kernel of dimension $n$, and $\mathbf{h}=(h_1,\cdots,h_n)$ are some data, which satisfy  the incompressibility condition $\sum_{i=1}^n h_{i,i}=0$. Here and in the following we use the classical Einstein notation in order to denote ordinary partial derivatives if this is convenient. The force terms $f_i,~ 1\leq i\leq n$ usually depend on time and space.  A more specific structure of the equation in (\ref{Navleray}) is revealed by the equivalent vorticity equation. In the case of dimension $n= 2$  the vorticity $\omega=(\omega_1,\omega_2,\omega_3)^T$ of the data $(v_1,v_2,0)^T$ is $(0,0,v_{2,1}-v_{1,2})$ such that the vorticity stretching term $v\cdot \nabla \omega$ vanishes and we are left with a scalar Burgers type equation of the form
\begin{equation}\label{dim2}
\frac{\partial}{\partial t}\omega_3+\sum_{j=1}^2v_j\omega_3=0.
\end{equation}

This enormous reduction of complexity, where loss of vorticity stretching may be linked to the loss of turbulent solutions in the informal sense described in \cite{LL}. Note that for $n=3$ incompressibility $\div \mathbf{v}=0$ and $\omega=\mbox{curl}(\mathbf{v})$ imply $\div \omega =0$. Indeed the example of a weakly singular Euler equation given below has no analogue in the case of two dimensions. In the following  we stick to the cases $n=3$ or $n\geq 3$.   

A recent abstract functional analytic approach is discussed in \cite{web} and reported in \cite{natletter}. Although it was shown that abstract functional analytical approaches which are based mainly on the energy identity are unikely to solve the problem of global uniqueness and regularity, the appoach which is published in
 \cite{O} has some merits. Actually, it is posed on a $n$-torus with zero Diriclet-Neumann boundary conditions and zero initial data $h_i=0,~ 1\leq i\leq n$, where $f_i\in L^2$ for $1\leq i\leq n$ is imposed for the external forces, a rather weak requirement. In its abstract formulation this is no restriction since the equation for $w_i:=v_i-h_i,~ 1\leq i\leq n$ can be subsumed under the abstract parabolic operator considered, at least if $h_i\in H^2$, where $H^2$ denotes the Sobolev space of order $2$ in $L^2$ theory. Boundary conditions for the resulting problem for $w_i,~ 1\leq i\leq n$ may be imposed  for $t>0$ keeping the content of the problem, and the resulting sufficiently regular boundary functions may again subsumed under the force terms. Hence, the problem considered in \cite{O} describes essentially a considerable  class of problems in the abstract form
\begin{equation}\label{af}
u_t+Au+B(u,u)=f,~u(0)=0,
\end{equation} 
or in the reduced form
\begin{equation}\label{rf}
\stackrel{\circ}{v}+L\left(\stackrel{\circ}{v},\stackrel{\circ}{v} \right)=\stackrel{\circ}{f},
\end{equation}
where $L\left(\stackrel{\circ}{v},\stackrel{\circ}{v} \right)=B\left(T\stackrel{\circ}{v},T\stackrel{\circ}{v} \right)$, $T\stackrel{\circ}{v}=\int_0^t\exp(-A(t-s))\stackrel{\circ}{v}(s)ds$, and
$\stackrel{\circ}{v}=u'+Au$ (cf. \cite{O} for notation of the abstract problem). A counterexample concerning the proposed stability of (\ref{rf}) is reported in \cite{web}, but this is in weak function spaces such that the author may hope to reestablish it in stronger function spaces. However, even so its formulation is definitely  too abstract (in the sense of forgetting special structure of the original Navier Stokes equations) in order to achieve a gobal existence and uniqueness result. One serious issue is whether the supercritical barrier fomulated in \cite{T} for stochastic averaged models can be avoided in the general setting of (\ref{af}). If there were a unique and smooth solution to the problem stated in \cite{F}, then  it would still be a major achievement if the result in \cite{T} could be transferred to the situation in (\ref{af}) with $f=0$ and (even regular) data $u(0)\in H^m\cap C^m$ for some $m\geq 2$, as this would show that some essential information of the more specific Navier Stokes equation is lost in the more abstract formulation. Here, as usual $C^m$ denotes the space of $m$-times continuously differentiable functions. Another definite issue of the problem formulation in \cite{O} are the time dependent force terms $f_i\in L^2$, which make it impossible to have uniqueness in general as we show in the next section. Paradoxically this is a merit of the formulation if compared to the requirement in \cite{F} that the forces are located in Schwartz space (even with respect to time). For we know that turbulence is caused by boundary conditions or by force terms, and the formulation in \cite{O} seems to be just the right formulation (with a consistent form of subsuming boundary conditions implicitly) if we want to set up an abstract class of models subsuming the phenomena of turbulence.
However, as we shall observe next it is indeed too abstract in order to have a deterministic model.

\section{A short proof of non-uniqueness for the Navier Stokes equation Cauchy problem with time dependent force terms}

Assuming $f_i\in L^2$ and simplifying initial-boundary conditions to zero seems rational, because in dimension $n=2$ it can be shown that two orders of regularity are gained for the source term. Initial boundary conditions $h_i\in H^2,~ 1\leq i\leq n$ are easily subsumed by the force terms of a related problem for $v_i-h_i,~ 1\leq i\leq n$, where non-zero Dirichlet or Neumann boundary conditions for $t>0$ may be set to zero and subsumed under the force terms as well. There may be a jump of the boundary conditions involved at initial time $t=0$, but such kind of problems can be solved by integral formulations - so they are not essential if they exist at all. In this case there may still be a hope that the existence of a global regular solution branch can be proved (if the supercritical barrier of the averaged model can be surpassed), but the solution cannot be unique in general as we show in the following.
% In order to avoid a technical mistake made in \cite{O} and discussed in \cite{web} concerning of the transfer of regularity of Poisson equation data $\sum_{j,k}v_{j,k}v_{k,j}$ to pressure, we introduce a function space which is a stronger than the function space $H^2$ with respect to decay at spatial infinity (for problems of bounded domains as considered in \cite{O} this means order of decay of Fourier mode coefficients). 
We consider the usual Cauchy problem on the whole domain of ${\mathbb R}^n$ here, with data functions $h_i,~ 1\leq i\leq n$ which 
have strong polynomial decay at infinity.  More precisely, for some $m\geq 2$ we consider some class of regular data with $h_i\in  {\cal C}^{m(n+1)}_{pol,m},~1\leq i\leq n$,
where for nonnegative integers $m,n,l$
\begin{equation}\label{pol1}
{\cal C}^{l}_{pol,m}:={\Big \{} f:{\mathbb R}^n\rightarrow {\mathbb R}: \\
\\
\exists c>0~\forall |x|\geq 1~\forall 0\leq |\gamma|\leq m~{\big |}D^{\gamma}_xf(x){\big |}\leq \frac{c}{1+|x|^l} {\Big \}}. 
\end{equation}
In order to have suitable regular spaces with some order of polynomial decay at spatial infinity at hand we define for nonnegative integers $m,l$
\begin{equation}
H^l_{pol,m}:=H^l\cap {\cal C}^{m(n+1)}_{pol,m},
\end{equation}  
where $H^l$ is the standard Sobolev space of order $l\geq 0$ ($L^2$-theory). Indeed, the term 'suitable' has a specific interpretation here, as it allows for transformation of regular problems for $m\geq 2$ to finite domains by coordinate transformation. We have noted elsewhere that this implies that classical limit procedures can be used which are stronger than the limits  obtained by using Rellich embedding for example. In this paper we avoid viscosity limits using different local iteration schemes and thereby we simplify arguments for the existence of regular local solution branches. 
 Now assume that on a local time interval $[0,T]$ for some $T>0$ a solution $v^E_i ,~ 1\leq i\leq n$ of the incompressible Euler equation, i.e., of the equation in (\ref{Navleray}) with $\nu=0$ and incompressibility of the data $\sum_{i=1}^nh^E_{i,i}=0$ is given, where $v^E_i(t,.)\in C^2\cap H^2_{pol,2}$ for $t\in [0,T]$ (especially $v^E_i(0,.)=h_i\in C^2\cap H^2_{pol,2}$). If $h_i\neq 0$ for some $1\leq i\leq n$, then such a solution $v^E_i,1 \leq i\leq n$ is also a nontrivial solution of a Navier Stokes equation on the domain $[0,T]\times {\mathbb R}^n$ with force terms
\begin{equation}
f^E_i:=-\nu\Delta v^E_i\in L^2,~t\in [0,T].
\end{equation}
Hence multiple regular solution branches of the incompressible Euler equation with data in $H^2_{pol,2}\cap C^2$ destroy the possibility of global regular unique solutions of Navier Stokes equation problems with time dependent force terms as in (\ref{af}). 

Now, for the latter task it is sufficient to observe
\begin{itemize}
\item[i)] for any given data $h_i\in H^2_{pol,2}\cap C^2,~ 1\leq i\leq n$ there is a short time solution
\begin{equation}
v^E_i\in C^{1,2}\left([0,T],{\mathbb R}^n\right) 
\end{equation}
of the Euler equation for some time $T>0$;
\item[ii)] for some data $h^{E,-}_i\in \left(  H^2_{pol,2}\cap C^{1,\alpha}\right) \setminus C^2,~ 1\leq i\leq n$ there is a short time solution $v^{E,-}_i\in \left( C^2\cap H^2_{pol,2}\right) \left((0,T], {\mathbb R}^n \right)$ of the time-reversed Euler equation, where $v^{E,-}_i,~ 1\leq i\leq n$ solves
\begin{equation}\label{Navleray1}
\left\lbrace \begin{array}{ll}
\frac{\partial v^{E,-}_i}{\partial \tau}
-\sum_{j=1}^n v^{E,-}_j\frac{\partial v^{E,-}_i}{\partial x_j}=\\
\\ \hspace{1cm}-\int_{{\mathbb R}^n}\left( \frac{\partial}{\partial x_i}K_n(x-y)\right) \sum_{j,m=1}^n\left( \frac{\partial v^{E,-}_m}{\partial x_j}\frac{\partial v^{E,-}_j}{\partial x_m}\right) (t,y)dy,\\
\\
\mathbf{v}^{E,-}(0,.)=\mathbf{h}^{E,-}=\left(h^{E,-}_1,\cdots,h^{E,-}_n\right)^T,
\end{array}\right.
\end{equation}
and where $\tau=T-t$.
Here for positive integer $m$ and H\"older exponent $\alpha\in (0,1)$  $C^{m,\alpha}$ is the space of functions with H\"older continuous derivatives of exponent $\alpha$ up to order $m$.
\end{itemize}
For it follows from i) and ii) that for the final data $h^E_i(.)=v^{E,-}_i(T,.),~ 1\leq i\leq n$ of a solution for the time-reversed equation in ii) we have $h^E_i\in C^2\cap H^2_{pol,2}$ for all $1\leq i\leq n$  and  a solution 
\begin{equation}\label{ve}
v^E_i\in C^2\left([0,T], {\mathbb R}^n \right),~v^E_i(0,.)=h^E_i(.) 
\end{equation}
according to i) such that the solution
\begin{equation}
t\rightarrow v^{E,2}_i(t,.):=v^{E,-}_i(\tau,.),~v^{E,2}_i(T,.)\in \left(H^2_{pol,2}\cap C^{1,\alpha}\right)\setminus C^2
\end{equation}
-which exists according to ii)- is necessarily different form $v^E_i,~ 1\leq i\leq n$ in (\ref{ve}) at $t=T$, while
\begin{equation}
v^{E,2}_i(0,.)=v^{E,-}_i(T,.)=h^E_i=v^E_i(0,.)\in C^2\cap H^2_{pol,2}. 
\end{equation}
\begin{rem}
Data in weaker function spaces can be found for the same conclusion in item ii), but the assumption given there is a) sufficient in order to prove non-uniqueness, and b) simplifies the proof of items i) and ii) below.
\end{rem} 
We consider for some $\nu >0$ a local time  iterative solution schemes $v^{(k)}_i,~1\leq i\leq n,~k\geq 0$ of (\ref{Navleray}) including forces 
\begin{equation}
f^{(k)}_i=-\nu  \Delta v^{(k-1)}\rightarrow f_i=-\nu  \Delta v_i,~ 1\leq i\leq n~ (as~ k\uparrow \infty).
\end{equation}
We choose $\nu=s$ and consider for $s>0$ families of representations in terms of fundamental solutions, 
\begin{equation}\label{Gsnu}
G^{s}(t,x;s,y)=\frac{1 }{\sqrt{4\pi s (t-s)}^n}\exp\left(-\frac{|x-y|^2}{4 s(t-s)} \right)
\end{equation}
 of the  heat equation $\frac{\partial p}{\partial t}-s\Delta p=0$ on the time intervals $[s,t]$
For $t>s\geq  0$ and $x\in {\mathbb R}^n$ we define a family of iteration schemes, where for at the first iteration step $(k)=(0)$ we define
\begin{equation}
v^{(0)}_i(t,x):=\int_{{\mathbb R}^n}h_i(y) G^{s}(t,x;s,y)dy=:h_i\ast_{sp}G^s(t,x),
\end{equation}
and for $k\geq 1$ we define 
\begin{equation}\label{Navleraysol}
\begin{array}{ll}
v^{(k)}_i
=h_i\ast_{sp}G^s-\sum_{j=1}^n \left( v^{(k-1)}_j v^{(k-1)}_{i,j}\right) \ast G^s\\
\\ +\int_{{\mathbb R}^n}\left( K_{n,i}(.-y)\right) \sum_{j,m=1}^n\left(  v^{(k-1)}_{m,j}v^{(k-1)}_{j,m}\right) (.,y)dy\ast G^s\\
\\
-s \Delta v^{(k-1)}_i\ast G^s,
\end{array}
\end{equation}
where for functions $g\in C([0,T],{\mathbb R^n})$ we denote $$g\ast G^s:=\int_s^t\int_{{\mathbb R}^n}g(\sigma,y)G^s(t,x;\sigma,y)dyd\sigma.$$
For multiindices $\gamma,\beta$ with $\beta_j=\gamma_j+\delta_{lj}$ for $1\leq j,l\leq n$ and $|\beta|=\sum_i\beta_i\geq 1$ and for $k\geq 1$ we define
\begin{equation}\label{Navleraysol2}
\begin{array}{ll}
D^{\beta}_xv^{(k)}_i
=D^{\gamma}h_i\ast_{sp}G^s_{,l}-\sum_{j=1}^n D^{\gamma}_x\left( v^{(k-1)}_j v^{(k-1)}_{i,j}\right) \ast G^s_{,l}\\
\\ +\int_{{\mathbb R}^n}\left( K_{n,i}(.-y)\right) \sum_{j,m=1}^nD^{\gamma}_x\left(  v^{(k-1)}_{m,j}v^{(k-1)}_{j,m}\right) (.,y)dy\ast G^s_{,l}\\
\\
-D^{\gamma}_x\left( s  \Delta v^{(k-1)}_i \ast G^s_{,l}\right) .
\end{array}
\end{equation}
In (\ref{Navleraysol}) derivatives may be shifted according to the usual convolution rule if convenient.
For sufficiently regular data $h_i,~ 1\leq i\leq n$ the function $v^{(k)}_i,~1\leq i\leq n$ in (\ref{Navleraysol}) solves the equation
\begin{equation}\label{Navleraysol3}
\begin{array}{ll}
\frac{\partial v^{(k)}_i}{\partial t}
-s\Delta v^{(k)}_i=-\sum_{j=1}^n  v^{(k-1)}_j v^{(k-1)}_{i,j}\\
\\ +\int_{{\mathbb R}^n}\left( K_{n,i}(.-y)\right) \sum_{j,m=1}^n\left(  v^{(k-1)}_{m,j}v^{(k-1)}_{j,m}\right) (.,y)dy
-s \Delta v^{(k-1)}_i,
\end{array}
\end{equation}
on an interval $[s,T]$ for some $T>s$, such that a regular fixed point limit ${\mathbf v}^{(k)}\uparrow {\mathbf v}^{E,s}$ solves the incompressible Euler equation on the time interval $[s,T]$. The time-shifted function ${\mathbf v}^{E}(.,.):={\mathbf v}^{E,s}(.+s,.)$ then solves the Euler equation on the time interval $[0,T-s]$. Since the transformation $t\rightarrow -t$ combined with an accompanying change $s\rightarrow -s$ does not affect the argument below, it is essential to construct a fixed point for some  data $h_i\in C^{1,\alpha}\setminus C^2,~ 1\leq i\leq n$. The short time fixed point argument for data $h_i\in C^2\cap H^2_{pol,2},~ 1\leq i\leq n$ then follows a fortioriy by similar arguments. We consider an example of data $h_i\in C^{1,\alpha}\setminus C^2\cap H^2_{pol,2},~ 1\leq i\leq n$ with $\sum_{i=1}^nh_{i,i}=0$ in the essential case $n=3$ and define for all $x\in {\mathbb R}^n$ and small $\epsilon >0$
\begin{equation}\label{geps}
g^{(\epsilon)}(x)=  r_2\cos\left(\frac{1}{r_2^{\epsilon}} \right)\frac{1}{\left(1+r_2^2 \right)^6},~ r_2:=\sqrt{x_1^2+x_2^2}.
\end{equation}
Here the upper script $(\epsilon)$ in brackets denotes an index.
\begin{equation}\label{h1h2h3}
h_1=x_1g^{(\epsilon)},~h_2=x_2g^{(\epsilon)},~h_3=-2x_3\left( g^{(\epsilon)}+\frac{1}{2}r_2\frac{d}{dr_2}g^{(\epsilon)}\right) ,
\end{equation} 
and where small epsilon ensures some regularity of the data. Note that  $h_{1,11}$ and $h_{1,22}$ have singular behavior $\sim r^{-2\epsilon}$ at $r$ close to zero, 
Incompressibility $\sum_{i=1}^nh_{i,i}=0$ holds at time $t=0$, and then holds -according to Leray- as long as a solution exists.
We note that $$\omega_3(0,.)=h_{1,2}(.)-h_{2,1}(.)=0$$ such that the two dimensional analogue of this example leads to a unique solution $\omega_3=0$ of equation (\ref{dim2}).

In order to construct a fixed point we prove contraction for the increments
\begin{equation}
\delta v^{(k)}_i:=v^{(k)}_i-v^{(k-1)}_i,~\mbox{where}~k\geq 1,~ 1\leq i\leq n.
\end{equation}
Next we observe that the linear terms in the equation for $v^{(k)}_i,~ 1\leq i\leq n$ in (\ref{Navleraysol3}) add up to
\begin{equation}\label{linterm}
s\Delta \left( v^{(k)}_i-v^{(k-1)}_i\right)=s\Delta \delta v^{(k)}_i
\end{equation}
for $k\geq 1$. We note that the latter expression has a time singulatity for $k=1$ of small order. However this singularity disappears for $k\geq 2$. We observe  
\begin{lem}
For small $s>0$ and  data $h_i,~ 1\leq i\leq n$ as above and for some $T>s$ the linear term of (\ref{linterm}) in
\begin{equation}\label{Navleraysol4}
\begin{array}{ll}
\frac{\partial v^{(k)}_i}{\partial t}
=-\sum_{j=1}^n  v^{(k-1)}_j v^{(k-1)}_{i,j}\\
\\ +\int_{{\mathbb R}^n}\left( K_{n,i}(.-y)\right) \sum_{j,m=1}^n\left(  v^{(k-1)}_{m,j}v^{(k-1)}_{j,m}\right) (.,y)dy
-s \Delta \delta v^{(k)}_i,
\end{array}
\end{equation}
converges to zero as $k\uparrow \infty$ on the interval $[s,T]$.
\end{lem}
\begin{proof} Note that equation (\ref{Navleraysol4}) is (\ref{Navleraysol3}) rewritten. Define
\begin{equation}
N^{(k)}_i=-\sum_{j=1}^n  v^{(k-1)}_j v^{(k-1)}_{i,j}\\
\\ +\int_{{\mathbb R}^n}\left( K_{n,i}(.-y)\right) \sum_{j,m=1}^n\left(  v^{(k-1)}_{m,j}v^{(k-1)}_{j,m}\right) (.,y)dy
\end{equation}
We have
\begin{equation}
\delta v^{(1)}_i=N^{(0)}_i\ast G^s+s\Delta v_i^{(0)}\ast G^s=N^{(0)}_i\ast G^s+s\Delta (h_i\ast_{sp}G^s)\ast G^s
\end{equation}
For some $T>s>0$ we have $s\Delta (h_i\ast_{sp}G^s)\ast G^s=s\sum_j(h_{i,j}G^s_{,j})\ast G^s_{,j}\in C_b^{0,2}((s,T],{\mathbb R}^n)$ (space of twice continuously differentiable functions with bounded spatial derivatives up to second order), where standard estimates of the Gaussian and upper bounds of elliptic integrals show that there is an upper bound  $\sim(t-s)^{-\epsilon}$ for small $\epsilon >0$ for $t>s$ close to $s$. For $k=2$ we have
\begin{equation}
\begin{array}{ll}
\delta v^{(2)}_i=(N^{(1)}_i-N^{(0)}_i)\ast G^s+s\Delta v_i^{(1)}\ast G^s\\
\\
=(N^{(1)}_i-N^{(0)}_i)\ast G^s\\
\\
+s\Delta \left(h_i\ast G^s+ N^{(0)}_i\ast G^s+s\Delta (h_i\ast_{sp}G^s)\ast G^s\right) \ast G^s
\end{array}
\end{equation}
Here, standard Gaussian upper bound estimates show that $\delta v^{(2)}_i\in C_b^{0,2}([s,T],{\mathbb R}^n)$ indeed. Inductively, we have
\begin{equation}
\delta v^{(k)}_i=(N^{(k-1)}_i-N^{(k-2)}_i)\ast G^s+s\Delta v_i^{(k-1)}\ast G^s
\end{equation}
and induction shows that $\lim_{k\uparrow\infty} \delta v^{(k)}_i=0$ if $\lim_{k\uparrow\infty}D^{\beta}_x(N^{(k-1)}_i-N^{(k-2)}_i)\ast G^s=0$ for $0\leq |\beta|\leq 2$. This is shown in the proof of the next lemma.
\end{proof}
For the preceding argument it is essential that $\delta N^{(k)}=N^{(k)}-N^{(k-1)}$ converges to zero as $k\uparrow \infty$. Having observed this it suffices to show that
contraction holds for the reduced increment which is defined by
\begin{equation}
\delta v^{*,(k)}_i=\delta v^{(k)}_i+s\Delta v^{(k)}_i\ast G^s,~ 1\leq i\leq n.
\end{equation}
A word concerning function spaces $C^2\cap H^2_{pol,2}$ is not a closed space, but it has interesting embeddings into closed subspaces: the strong polynomial decay implies that a transformation $x_i\rightarrow \arctan(x_i)$ leads to an embedding of transformed  subspaces of a closed space $C^2_b(\Omega)$ of twice continuously differentiable functions with bounded spatial derivatives up to order $2$ and zero data on the boundary of the square $\Omega=\left( -\frac{\pi}{2},\frac{\pi}{2}\right) ^n$. For our purposes it is sufficient to define
\begin{equation}
{\big |}f{\big |}_{C^2\cap H^2_{pol,2}}:=1_{f\in C^2\cap H^2_{pol,2}}\sum_{0\leq |\beta|\leq 2}\sup_{y\in {\mathbb R}^n}{\big |}D^{\beta}_xf(y){\big |},
\end{equation}
where $1_{f\in C^2\cap H^2_{pol,2}}$ is the characteristic function of $C^2\cap H^2_{pol,2}$.
We have
\begin{lem}
For $h_i,~ 1\leq i\leq n$ as in (\ref{h1h2h3}) above with $\epsilon >0$ small enough
the sequence of nonlinear increments $\left( \delta N^{(k)}\ast G^s\right)_{k\geq 2}$ converges to zero in the function space $C^2\cap H^2_{pol,2}$ on the time interval $[s,T]$.
Furthermore, there exists $T>0$ such that for all $0<s<T$ and all $k\geq 0$
\begin{equation}\label{assertcon}
\sup_{t\in [s,T]}{\big |}\delta v^{*,(k+1)}_i(t,.){\big |}_{C^2\cap H^2_{pol,2}}\leq \frac{1}{2}\sup_{t\in [s,T]}{\big |}\delta v^{*,(k)}_i(t,.){\big |}_{C^2\cap H^2_{pol,2}}.
\end{equation}
\end{lem}
\begin{proof}
We have $v^{(0)}_i=h_i\ast_{sp}G^s$ with $h_i, 1\leq i\leq n$ as in (\ref{h1h2h3}) above, and
observe that
\begin{equation}
v^{(0)}_{i,jk}=h_{i,j}\ast_{sp}G^s_{,k},
\end{equation}
where $h_{i,j}\in C^{\alpha}$ with $\alpha$ close to one if $\epsilon >0$ is small,  where an upper bound can be obtained via
\begin{equation}
|(h_i(x)-h_i(y))G^s_{,k}(t,x;s,y)|\leq \frac{C}{(t-s)^{\delta}|x-y|^{n+1-2\delta-\alpha}}
\end{equation}
with a finite constant $C>0$ (the subtracted term is added seperately and treated by a surface integral obtained by the divergence theorem). Due to large $\alpha$ the latter term is integrable even for small $\delta >\frac{1-\alpha}{2}$
 Consider first the Burgers term in $N^{(0)}_i$. It is of the form
\begin{equation}
B^{(0)}_i\ast G^s:=\left( \sum_{j=1}^nv^{(0)}_jv^{(0)}_{i,j}\right)\ast G^s=\left( \sum_{j=1}^n(h_j\ast_{sp}G^s)(h_i\ast_{sp}G^s_{,j})\right)\ast G^s.
\end{equation}
Second order spatial derivatives of $B_{i,k}\ast G^s_{,l}$ can be estimated by terms of the form
\begin{equation}\label{burg1}
{\Bigg |}\left( \sum_{j=1}^n(h_j\ast_{sp}G^s)(h_{i,k}\ast_{sp}G^s_{,j})\right)\ast G^s_{,l}{\Bigg |}
\end{equation}
and similar terms with even weaker singulatities of the factors. Time singularities of upper bounds of  ${|}G^s_{,l}{|}$ are of the form $C_1t^{-\delta_0}$ where $\delta_0\in (0.5,1)$ can be chosen such that expressions as (\ref{burg1}) are integrable with respect to time (usual elliptic integral estimates).
Similarly we observe that the Leray projection term is twiche diffentiable with respect to the spatial variables and continuous with respect to the time variable on the time interval $[s,t]$, i.e.,
\begin{equation}
L^{(0)}_i\ast G^s:=K_{,i}\ast_{sp}\left( \sum_{j,l=1}^n(h_j\ast_{sp}G^s_{,l})(h_l\ast_{sp}G^s_{,j})\right)\ast G^s\in C^{0,2}.
\end{equation}
The only singular terms at time $t=s$ which enters the construction of $v^{(1)}_i,1 \leq i\leq n$ are the second order spatial derivatives of the linear term $s\Delta v^{(0)}_\ast G^s=s(\sum_{i=1}^nh_{i,i}\ast_{sp}G^s_{,i})\ast G^s$. Since $h_{i,i}\in C^{\alpha}$ for all $1\leq i\leq n$ we know that $h_{i,i}\ast G^s_{,i}$ is spatially differentiable, hence spatially Lipschitz. This Lipschitz continuity implies  that $s(\sum_{i=1}^nh_{i,i}\ast_{sp}G^s_{,i})\ast G^s_{,jk}$ is spatially integrable where the related elliptic convolution intergal with respect to time of the form $\int_{s}^t\frac{C_2}{\sigma^{\delta_0}(t-\sigma)^{\delta_1}}$ for some finite constant $C_2>0$ and $\delta_0,\delta_1>0.5$,nut close to $0.5$ of $\alpha$ is close to $1$. Hence, $s\Delta v^{(0)}_i\ast G^s_{,jk}$ has a weakly singular upper bound of form $C_3t^{1-\delta_1-\delta_2}$ for some finite constant $C_3>0$.
 These terms also enter into the construction of $v^{(2)}_i,~ 1\leq i\leq n$. However, similar analysis of $B^{(1)}_i\ast G^s$ and $L^{(1)}_i\ast G^s$ show that these terms are indeed in $C^{0,2}$ on $[s,t]\times {\mathbb R}^n$ and so is $N^{(2)}_i$ and inductively $N^{(k)}_i$ for $k\geq 2$ and $1\leq i\leq n$. Note that the multiplicative form of the nonlinear terms (which is only slightly weakenend by convolutions 
with the Gaussians or spatial derivatives up to second order of the Gaussians ensures that $N^{(k)}_i\in H^2_{pol,2}$ for $1\leq i\leq n$ and $k\geq 0$. The convergence of $(\delta N^{(k)}_i)_{k\geq 2}$ to zero in $C^{0,2}$ then follows from contraction of the reduced increments $\delta v^{*,(k)}_i$, which we show next. 
For given $s>0$ and $T$ small enough (to be determined) we consider the iteration scheme for given $t\in (s,T]$ and construct an upper bound for 
$\max_{1\leq i\leq n}{\big |} D^{\beta}_x\delta v^{(k)}_i(t,.){\big |}$ and for $|\beta|\geq 1$ and 
$k\geq 2$, where we extract the functional increment $\max_{1\leq i\leq n}{\big |} D^{\beta}_x \delta v^{(k-1)}_i(t,.){\big |}$. For $|\beta|\geq 1$ we have

\begin{equation}\label{derest}
\begin{array}{ll}
D^{\beta}_x\delta v^{*,(k)}_i
=-\sum_{j=1}^n D^{\gamma}_x\left(\delta  v^{*,(k-1)}_j v^{*,(k-1)}_{i,j}\right) \ast G^s_{,l}\\
\\
-\sum_{j=1}^n D^{\gamma}_x\left( v^{*,(k-1)}_j \delta v^{*,(k-1)}_{i,j}\right) \ast G^s_{,l}\\
\\ 
+2\int_{{\mathbb R}^n}\left( K_{n,i}(.-y)\right) \sum_{j,m=1}^nD^{\gamma}_x\left(  \delta  v^{*,(k-1)}_{m,j}v^{*,(k-1)}_{j,m}\right) (.,y)dy\ast G^s_{,l},
\end{array}
\end{equation}
where the factor $2$ is by the symmetry of the Leray projection term. The first order derivatives of fundamental solutions $G^*_{,l}$ have an upper bound
\begin{equation}
\begin{array}{ll}
{\big |}G^s_{,l}(x,t;y,s){\big |}_{L_1([s,T]\times {\mathbb R}^n)}\leq \int_s^t\int_{B_1(x)}\frac{c}{|\sigma-s|^{\delta}|z-y|^{n+1-2\delta}}dzd\sigma\\
\\
\hspace{2.8cm}+\int_s^t\int_{{\mathbb R}^n\setminus B_1(x)}{\big |}G^s_{,l}(z,\sigma;y,s){\big |}dzd\sigma\leq CT^{1-\delta}
\end{array}
\end{equation}
for $\delta \in (0.5,1)$ and some finite constants $c,C>0$, and where $B_1(x)$ is the $n$-dimensional ball of radius $1$ around $x$. Note that the second integral becomes small as $T>t$ becomes small. Hence the integrals in (\ref{derest}) abbreviated by the symbol $\ast$ can be estimated by  normal convolutions invoking Young's inequality such that
\begin{equation}\label{derest2}
\begin{array}{ll}
{\big |}D^{\beta}_x\delta v^{*,(k)}_i{\big |}_{L^2_C\cap H^0_{pol,2}}
\leq CT^{1-\delta}{\Big (}\sum_{j=1}^n{\big |} D^{\gamma}_x\left(\delta  v^{*,(k-1)}_j v^{*,(k-1)}_{i,j}\right){\big |}_{L^2_C\cap H^0_{pol,2}} \\
\\
+\sum_{j=1}^n {\big |}D^{\gamma}_x\left( v^{*,(k-1)}_j \delta v^{*,(k-1)}_{i,j}\right){\big |}_{L^2_C\cap H^0_{pol,2}}\\
\\ 
+2\int_{{\mathbb R}^n}{\big |}\left( K_{n,i}(.-y)\right) \sum_{j,m=1}^nD^{\gamma}_x\left(  \delta  v^{*,(k-1)}_{m,j}v^{*,(k-1)}_{j,m}\right) (.,y)dy{\big |}_{L^2_C\cap H^0_{pol,2}}{\Big )},
\end{array}
\end{equation}
where $L^2_C=L^2\cap C^0$ is the space of continuous functions $C^0$ which are in $L^2$ and where $T>0$ is small enough such that the  last linear term according to our observations above (absorbing the factor $t$ and one first order spatial derivative into the constant $C$). Assume inductively that for $k\geq 1$
\begin{equation}\label{cl}
\sum_{0\leq |\delta|\leq 2}\sup_{\sigma\in [s,T]}
{\big |}D^{\delta}v^{*,(k-1)}_i(\sigma,.){\big |}_{L^{\infty}_C\cap H^0_{pol,2-\epsilon}}\leq C_0+\sum_{l=1}^{k-1}\frac{1}{4^l},
\end{equation}
where the latter sum is considered to be zero for $k=1$, and where $L^{\infty}_C=C^0\cap L^{\infty}$ is the space of bounded continuous functions. Then we have
\begin{equation}\label{derest3}
\begin{array}{ll}
{\big |}D^{\beta}_x\delta v^{*,(k)}_i{\big |}_{L^2\cap H^0_{pol,2-\epsilon}}
\leq 2CT(C_0+1)(n^2+n)(C_K+1) \times\\
\\
\times{\Bigg (}\max_{1\leq j\leq n}{\big |} D^{\gamma}_x\left(\delta  v^{*,(k-1)}_j \right){\big |}_{L^2\cap H^0_{pol,2-\epsilon}} +\max_{1\leq i\leq n} {\big |}D^{\beta}_x\delta v^{*,(k-1)}_{i}{\big |}_{L^2\cap H^0_{pol,2-\epsilon}}\\
\\ 
+2\max_{1\leq m\leq n}{\big |} D^{\beta}_x \delta  v^{*,(k-1)}_{m}(.){\big |}_{L^2\cap H^0_{pol,2-\epsilon}}{\Bigg )}.
\end{array}
\end{equation}
There are $4$ terms on the right side of the latter inequality, where one term has the factor 2. Analogous estimates hold for first order spatial derivatives and the value function itself. Furthermore there are $n^2+n+1$ multiindices $\beta$ with $0\leq |\beta|\leq 2$. Noting that $T^{1-\delta}>T$ for $T,\delta\in (0,1)$ we choose
\begin{equation}
T=\frac{1}{40(n^2+n+1)CT(C_0+1)(n^2+n)(C_K+1) }
\end{equation}
in order to have a contraction constant smaller or equal to $\frac{1}{4}$ in (\ref{assertcon}). This choice apperantly closes also the induction step for the upper bound in (\ref{cl}).
The argument for $|\beta|=0$ is similar and the proof is finished.
\end{proof}

\section{Conclusion}
Uniqueness does not hold for the general Navier Stokes model (GNSM) described in (\ref{af}) or (\ref{rf}), even in strong function spaces. Singularities may even arise in reduced forms of GNSM (e.g. for time-independent forces) if techniques in \cite{T} can be applied. For turbulence modeling the generality of GNSM may be a minimal request, because boundary conditions can be subsumed by external forces essentially, and in this context $L^2$-forces are suitable. As a consequence we may have either not well-defined or indeterministic models which encode turbulence, or unsuitable models which allow us to prove global regular existence and uniqueness.  A partial remedy may be that a proof of global regular solution branches may still be possible for GNSM in stronger function spaces, where damping estimates hold (and even damping estimates by homotopy as in \cite{O} are not excluded), and where the mistake outlined in \cite{web} is avoided. A proof of singularities for the simple model in \cite{F} would even sharpen this dilemma. However, it seems that this does not hold. Incompressibility implies that the solution increment $\delta v_i=v_i-h_i\ast_{sp}G_{\nu}$ can be written in terms of convolutions with the first order spatial derivatives of the Gaussian, i.e., on some time interval $[t_0,t_0+\Delta],~t_0\geq 0,\Delta >0$ 
\begin{equation}\label{Navleraysol2}
\begin{array}{ll}
\delta v_i
=\sum_{j=1}^n\left( v_j v_i\right) \ast G_{\nu,j}\\
\\
+\sum_{j,m=1}^n\int_{{\mathbb R}^n}\left( K_n(.-y)\right) \sum_{j,m=1}^n\left( v_{m,j}v_{j,m}\right) (.,y)dy\ast G_{\nu,i},~ 1\leq i\leq n.
\end{array}
\end{equation}
where $G_{\nu}$ is the fundamental solution of $\frac{\partial}{\partial t}p-\nu \Delta p=0$.
This increment is smaller than a corresponding increment  in the general Navier Stokes model (GNSM) in(\ref{af}) or (\ref{rf}) due to the antisymmetry $y_i\rightarrow -y_i$ of $G_{\nu,i}$ if the the Lipschitz continuity of the convoluted Burgers term and Leray projection term in the convolution in (\ref{Navleraysol2}) is ensured. In strong space the increment
upper bound
\begin{equation}\label{parafreeest}
 {\Big |} L\int_{0}^{\Delta}   \frac{1}{4} \frac{1}{\sqrt{4\pi}^n}\frac{1}{\sigma^{2.5}}\exp\left(-\frac{1}{4\sigma} \right)d\sigma {\Big |}\in o(\Delta),
\end{equation}
indicates a linear time upper bound, and it seems that it can be offset by autocontrol and damping estimates. For a suitable turbulence model such as  GNSM with the abstract scheme in (\ref{af}) or (\ref{rf}) the corrresponding increments are of order $O(\Delta)$ (big O). If a partial remedy of this dilemma is possible, namely that this growth can be offset by damping effects, is an open question.

\end{document}